\documentclass[reqno]{amsart}

\usepackage[bookmarks, colorlinks=true, pdfstartview=FitV, linkcolor=blue, citecolor=blue, urlcolor=blue]{hyperref}

\usepackage{general,resistance}
\usepackage[left=1.4in,right=1.4in,top=1in,bottom=1in]{geometry}
\usepackage{latexsym}
	\usepackage[pagewise]{lineno}
	\usepackage[dayofweek]{datetime} \settimeformat{ampmtime} \usdate
\usepackage[compress]{cite}  

\renewcommand{\linenopax}{~\par\vspace{-2mm}} 

\newcommand{\LapK}{{\ensuremath{\Lap_{\energy}^{(Kr)}}}\xspace}     
\newcommand{\dgl}{\mspace{3mu}\operatorname{d}\negsp[2]\gl}
\newcommand{\dgm}[1][]{\mspace{3mu}\operatorname{d}\negsp[2]\gm_{#1}}

\numberwithin{equation}{section}
\numberwithin{theorem}{section}

\title[Symmetric pairs and self-adjoint extensions of operators]{Symmetric pairs and self-adjoint extensions of operators, with applications to energy networks}

\author{Palle E. T. Jorgensen and Erin P. J. Pearse} 
\address{University of Iowa, Iowa City, IA 52246-1419 USA \\ \qq\texttt{palle-jorgensen@uiowa.edu}}

\address{California Polytechnic University, San Luis Obispo, CA 93407-0403 USA \\ \qq\texttt{epearse@calpoly.edu}}

\date{\textbf{Version of \textbf{\currenttime} on  \textbf{\longdate{\today}}}}



\begin{document}

\maketitle

\begin{abstract}
We provide a streamlined construction of the Friedrichs extension of a densely-defined self-adjoint and semibounded operator $A$ on a Hilbert space $\mathcal{H}$, by means of a symmetric pair of operators. A \emph{symmetric pair} is comprised of densely defined operators $J: \mathcal{H}_1 \to \mathcal{H}_2$ and $K: \mathcal{H}_2 \to \mathcal{H}_1$ which are compatible in a certain sense. With the appropriate definitions of $\mathcal{H}_1$ and $J$ in terms of $A$ and $\mathcal{H}$, we show that $(JJ^\star)^{-1}$ is the Friedrichs extension of $A$. Furthermore, we use related ideas (including the notion of unbounded containment) to construct a generalization of the construction of the Krein extension of $A$ as laid out in a previous paper of the authors. These results are applied to the study of the graph Laplacian on infinite networks, in relation to the Hilbert spaces $\ell^2(G)$ and $\mathcal{H}_{\mathcal E}$ (the energy space).
\end{abstract}

\allowdisplaybreaks

\section{Introduction}

Motivated by Laplace operators on infinite networks and their self-adjoint extensions, we consider the situation in which a certain two different Hilbert spaces contain a common linear subspace: $V \ci \sH_1 \cap \sH_2$. We study (possibly unbounded) from $\sH_i$ to $\sH_j$ in terms of whether or not $V$ is dense in one or both Hilbert spaces. In particular, we introduce the notion of a symmetric pair (see Definition~\ref{def:closable-pair}) of operators: when the densely defined operators $A: \sH_1 \to \sH_2$ and $B: \sH_2 \to \sH_1$ are compatible (i.e., there exists a suitable relation with their adjoints), then we immediately obtain that both are closable: see Lemma~\ref{thm:closable-pair}. In the present context, this can be applied to the operator $J:V \to \sH_2$ defined by $J\gf = \gf$, which (as function on sets) is the inclusion map. This provides for a very concise construction of the Friedrichs extension of a semibounded operator $A:\dom(A) \ci \sH \to \sH$. We use $A$ to define a new and strictly finer topology on \sH so that $J:V \to \sH$ is a contractive (inclusion) embedding, and then the key result Theorem~\ref{thm:Fried} yields the Friedrichs extension as $A_\sF = (JJ^\ad)^{-1}$.
Our next main results is Theorem~\ref{thm:Krein-generalization}, in which we leverage symmetric pairs to prove a generalization of Krein's extension results. In a forthcoming work, we use these ideas to describe a construction of the Krein extension \cite{Krein}, applications to reflection positivity in physics \cite{DualityRP}, construct a noncommutative analogue of the Lebesgue--Radon--Nikodym decomposition \cite{Charp} (see also \cite{Krein} and \cite{DualityRP}), and also to verify closability and compute adjoints of unbounded operators arising in the context of stochastic calculus (Malliavin derivative) and the study of von Neumann algebras (Tomita-Takesaki theory) \cite{SPandGHS}.

We further apply the results described above to discrete Laplace operators on infinite networks. Here, a network is just an connected undirected weighted graph $(G, c)$ (see Definition~\ref{def:network}), and the associated network Laplacian \Lap acts on functions $u:G \to \bR$; see Definition~\ref{def:graph-laplacian}. 
We restrict attention to the case when the network is \emph{transient}%
	\footnote{This is equivalent to assuming the existence of \emph{monopoles}; see Definition~\ref{def:dipole} and Remark~\ref{rem:monotransience}.}, 
	and we are particularly interested in the case when \Lap is unbounded, in which case some care must be taken with the domains. We consider \Lap separately as an operator on \HE, the Hilbert space of finite energy functions on $G$ and as on operator on $\ell^2(G)$. 
	Although the two operators agree formally, their spectral theoretic properties are quite different.
    The space \HE is defined in terms of the quadratic form \energy, which gives the Dirichlet energy of a function $u$; see Definition~\ref{def:H_energy}. By $\ell^2(G)$, we mean the unweighted space of square-summable functions on $G$ under counting measure; see Definition~\ref{def:ell2}. 

Neither of the two Hilbert spaces is contained in the other, and the two Hilbert norms do not compare. It follows that the spectral theory is quite different for the corresponding incarnations of the Laplacian: as an operator on $\ell^2(G)$ (Definition~\ref{def:domLap2}) and as an operator on \HE (Definition~\ref{def:domLapE}). We will use the respective notation $\Lap_2$ and \LapE to refer to these two very distinct incarnations of the Laplacian. Common to the two is that each is defined on its natural dense domain in each of the Hilbert spaces $\ell^2(G)$ and \HE, and in each case it is a Hermitian and non-negative operator. 
However, it is known from \cite{SRAMO} (see also \cite{Woj07, KellerLenz09, KellerLenz10}) that \Lap is essentially self-adjoint on its natural domain in $\ell^2(G)$ but in \cite{SRAMO} it is shown that \Lap is \emph{not} essentially self-adjoint on its natural domain in \HE (see Definition~\ref{def:domLapE}). Nonetheless, we prove that the Friedrich extension of the latter has a spectral theory that can be compared with the former. 

\subsection{Historical context and motivation}
    The importance of the Friedrichs extension of an unbounded Hermitian operator on a Hilbert space stems from its role in the classical theory. The network Laplace operator considered in this article is a discrete analogue of the better known Laplace operator associated to a manifold with boundary in harmonic analysis and PDE theory, see for example \cite{Fr35, Fr39} and the endnotes of \cite[Ch.~XII]{DuSc88}. In classical applications from mathematical physics, this Laplacian is an unbounded operator initially defined on a domain of smooth functions vanishing on the boundary. To get a self-adjoint operator in $L^2$ (and an associated spectral resolution), one then assigns boundary conditions. Each distinct choice yields a different self-adjoint extension (realized in a suitable $L^2$-space). The two most famous such boundary conditions are the Neumann and the Dirichlet conditions. In the framework of unbounded Hermitian operators in Hilbert space, the Dirichlet boundary conditions correspond to a semibounded self-adjoint extension of \Lap called the Friedrichs extension. For boundary value problems on manifolds with boundaries, the Hermitian property comes from a choice of a minimal domain for the given elliptic operator $T$ under consideration, and the semiboundedness then amounts to an a priori coercivity estimate placed as a condition on $T$. 

Today, the notion of a Friedrichs extension is typically understood in a more general operator theoretic context concerning semibounded Hermitian operators with dense domain in Hilbert space, see e.g., \cite[p.1240]{DuSc88}. In its abstract Hilbert space formulation, it throws light on a number of classical questions in spectral theory, and in quantum mechanics, for example in the study of Sturm-Liouville operators and Schr\"odinger operators, e.g., \cite{Kat95}. If a Hermitian operator is known to be semibounded, we know by a theorem of von Neumann that it will automatically have self-adjoint extensions.%
\footnote{The term ``extension'' here refers to containment of the respective graphs of the operators under consideration.}
  The selection of appropriate boundary conditions for a given boundary value problem corresponds to choosing a particular self-adjoint extension of the partial differential operator in question. In general, some self-adjoint extensions of a fixed minimal operator $T$ may be semibounded and others not. \emph{The Friedrichs extension is both self-adjoint and semibounded, and with the same lower bound as the initial operator $T$ (on its minimal domain).}

We are here concerned with a different context: analysis and spectral theory of problems in discrete context,
  wherein \Lap is the infinitesimal generator of the random walk on $(G,c)$.   In this regard, we are motivated by a number of recent papers, some of which are cited  above.
  A desire to quantify the asymptotic behavior of such reversible Markov chains leads to the need for precise and useful notions of boundaries of infinite graphs. Different conductance functions lead to different Laplacians \Lap, and also to different boundaries. In the energy Hilbert space \HE, this operator \Lap will then have a natural dense domain turning it into a semibounded Hermitian operator, and as a result, Friedrichs' theory applies.
%
As in classical Riemannian geometry, one expects an intimate relationship between metrics and associated Laplace operators. 
This is comparable to the use of the classical Laplace operator in the study of manifolds with boundary, or even just boundaries of open domains in Euclidean space, see e.g., \cite{Fug91, Fug86}.

\section{Symmetric pairs}

Self-adjoint extensions of unbounded operators may be studied via symmetric pairs. See also \cite{Charp, SPandGHS, Krein, DualityRP} for further applications of symmetric pairs to the closability of operators and computation of adjoints.

\begin{defn}\label{def:closable-pair}
	Suppose $\sH_1$ and $\sH_2$ are Hilbert spaces and $A,B$ are operators with dense domains $\dom A \ci \sH_1$ and $\dom B \ci \sH_2$ and
	\begin{align*}
		A: \dom A \ci \sH_1 \to \sH_2
		\qq\text{and}\qq
		B: \dom B \ci \sH_2 \to \sH_1.
	\end{align*}
	We say that $(A,B)$ is a \emph{symmetric pair} iff 
	\begin{align}\label{eqn:closable-pair}
		\la A \gf, \gy \ra_{\sH_2} = \la \gf, B\gy \ra_{\sH_1},
		\qq\text{for all } \gf \in \dom A, \gy \in \dom B. 
	\end{align}
	In other words, $(A,B)$ is a symmetric pair iff $A \ci B^\ad$ and $B \ci A^\ad$.
\end{defn}

\begin{lemma}\label{thm:closable-pair}
	If $(A,B)$ is a symmetric pair, then $A$ and $B$ are each closable operators. Moreover, 
	\begin{enumerate}
	\item $A^\ad \cj A$ is densely defined and self-adjoint with $\dom A^\ad \cj A \ci \dom \cj A \ci \sH_1$, and 
	\item $B^\ad \cj B$ is densely defined and self-adjoint with $\dom B^\ad \cj B \ci \dom \cj B \ci \sH_2$. 
	\end{enumerate}
\end{lemma}
\begin{proof}
	Since $A$ and $B$ are densely defined and $A \ci B^\ad$ and $B \ci A^\ad$, it is immediate that $A^\ad$ and $B^\ad$ are densely defined; it follows by a theorem of von Neumann that $A$ and $B$ are both closable. 
	By another theorem of von Neumann,  $A^\ad \cj A$ is self-adjoint; cf.~\cite[Thm.~13.13]{Rud91}.
\end{proof}

\begin{remark}
	Observe that by Lemma~\ref{thm:closable-pair}, there is a partial isometry $V:\sH_1 \to \sH_2$ such that $\cj A = V(A^\ad\cj A)^{1/2} = (B^\ad \cj B)^{1/2} V$. In particular, 
	\begin{align*}
		\spec_{\sH_1} (A^\ad \cj A) \less \{0\} = \spec_{\sH_2} (B^\ad \cj B) \less \{0\}.
	\end{align*}
\end{remark}

\begin{remark}
	Whenever $(A,B)$ is a symmetric pair, we may now assume (by Lemma~\ref{thm:closable-pair}) that $A$ and $B$ are closed operators. In the sequel, we will thus refer to the self-adjoint operators $A^\ad A$ and $B^\ad B$.
\end{remark}

The following example illustrates the relationship that can exist between the adjoint of an operator between $L^2$ spaces, and the Radon-Nikodym derivative of their respective measures. We return to this theme in Example~\ref{exm:RN} and in the forthcoming work \cite{Charp}; see also \cite{Jor80b, Ota88, Hassi07}

\begin{exm}\label{exm:trivial-adjoint}
	Let $X=[0,1]$, and consider $L^2(X,\gl)$ and $L^2(X,\gm)$ for measures \gl and \gm which are mutually singular. For concreteness, let \gl be Lebesgue measure, and let \gm be the classical singular continuous Cantor measure. Then the support of \gm is the middle-thirds Cantor set, which we denote by $K$, so that $\gm(K)=1$ and $\gl(X\less K)=1$. The continuous functions $C(X)$ are a dense subspace of both $L^2(X,\gl)$ and $L^2(X,\gm)$ (see, e.g. \cite[Ch.~2]{Rud87}). Define the ``inclusion'' operator\footnote{As a map between sets, $J$ is the inclusion map $C(X) \hookrightarrow L^2(X,\gm)$. However, we are considering $C(X) \ci L^2(X,\gl)$ here, and so $J$ is not an inclusion map between Hilbert spaces because the inner products are different. Perhaps ``pseudoinclusion'' would be a better term.} $J$ to be the operator with dense domain $C(X)$ and
\linenopax
\begin{align}\label{eqn:L2-inclusion-exm}
	J:C(X) \ci L^2(X,\gl) \to L^2(X,\gm)
	\qq\text{by}\qq
	J \gf = \gf.
\end{align}
We will show that $\dom J^\ad = \{0\}$, so suppose $f \in \dom J^\ad$. Without loss of generality, one can assume $f \geq 0$ by replacing $f$ with $|f|$, if necessary.
By definition, $f \in \dom J^\ad$ iff there exists $g \in L^2(X,\gl)$ for which
\linenopax
\begin{align}\label{eqn:adj-exm-defn}
	\la J\gf,f\ra_{\gm} = \int_X \cj\gf f \dgm = \int_X \cj\gf g \dgl = \la \gf, g \ra_{\gl},
	\qq \text{for all }\gf \in C(X). 
\end{align}
One can choose $(\gf_n)_{n =1}^\iy \ci C(X)$ so that $\gf_n|_{K}=1$ and $\lim_{n \to \iy} \int_X \gf_n \dgl = 0$ by considering the appropriate piecewise linear modifications of the constant function 1. For example, see Figure~\ref{fig:cantor-collapse}.
	\begin{figure}[b]
		\begin{centering}
		\scalebox{0.7}{\includegraphics{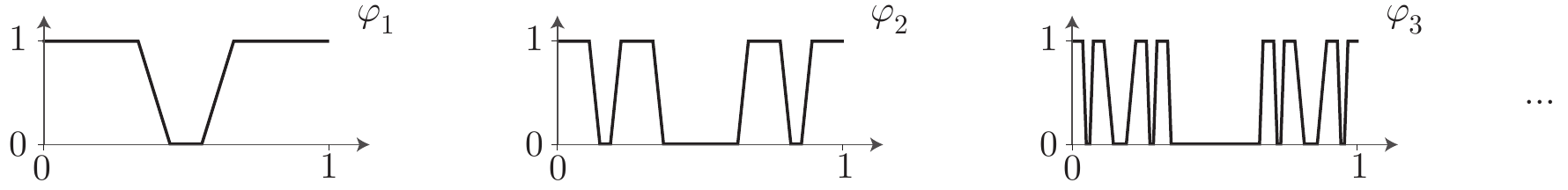}} 
		\end{centering}
		\caption{\captionsize A sequence $\{\gf_n\} \ci C(X)$ for which $\gf_n|_K=1$ and $\lim \int_X \gf_n \dgl =0$. See Example~\ref{exm:trivial-adjoint}.}
		\label{fig:cantor-collapse}
	\end{figure}
Now we have 
\linenopax
\begin{align}\label{eqn:adj-exm-defn}
	\la \gf_n, J^\ad f \ra_{\gl}
	= \la \gf_n, f \ra_{\gm} 
	= \la 1, f \ra_{\gm} 
	= \int_X |f| \dgm,
	\qq\text{for any } n,
\end{align}
but $\lim_{n \to \iy} \int_X \gf_n g \dgl = 0$ for any continuous $g \in L^2(X,\gl)$. 
Thus $\int_X |f| \dgm = 0$, so that $f = 0$ \gm-a.e.
In other words, $f =0 \in L^2(X,\gm)$ and hence $\dom J^\ad = \{0\}$, which is certainly not dense! Thus, one can interpret the adjoint of the inclusion as multiplication by a Radon-Nikodym derivative (``$J^\ad f  = f\frac{\dgm}{\dgl}$''), which must be trivial when the measures are mutually singular. This comment is made more precise in Example~\ref{exm:RN} and Corollary~\ref{thm:RN-connection}.
As a consequence of this extreme situation, the inclusion operator in \eqref{eqn:L2-inclusion-exm} is not closable.
\end{exm}

\section{The Friedrichs extension}
\label{sec:Fried-intro}

For a large class of symmetric operators, there is a canonical choice for a self-adjoint extension, the Friedrichs extension. 

\begin{remark}
The importance of the Friedrichs extension of an unbounded Hermitian operator on a Hilbert space stems from its role in the classical theory (and mathematical physics in particular). 
For example, consider the Laplace operator \gD defined initially on $C_0^\iy(\gW)$, where \gW is a regular open subset of \bRn for which $\cj{\gW}$ is compact. Thus, $\dom \gD$ consists of smooth functions vanishing at the boundary of \gW. To get a self-adjoint operator in the Hilbert space $\sH=L^2(\gW)$ (and an associated spectral resolution), one then assigns boundary conditions; each distinct choice yields a different self-adjoint extension. The two most famous choices of boundary conditions are the Neumann and the Dirichlet conditions. 

The Friedrichs extension procedure may be described abstractly, as the Hilbert completion of $\dom \gD$ with respect to a quadratic form defined in terms of \gD; cf.~\cite{Kat95, DuSc88}. Nonetheless, in the present example, the Friedrichs extension turns out to correspond to Dirichlet boundary conditions. (The Krein extension may also be defined abstractly, in terms of $\ker \gD^\ad$, turns out to correspond to Neumann conditions.) 
\end{remark}

	Consider an operator $A:\dom A \ci \sH \to \sH$, whose domain is dense in the Hilbert space \sH, and assume $A$ satisfies
	\linenopax
	\begin{align}\label{eqn:coercivity}
		\la \gf, A \gf \ra \geq \|\gf\|^2, \qq \text{for all } \gf \in \dom A.
	\end{align}
	Define $\sH_A$ to be the Hilbert completion of $\dom A$ with respect to the norm induced by
	\linenopax
	\begin{align}\label{eqn:A-norm}
		\la \gy, \gf \ra_A := \la \gy, A\gf\ra, \qq \gy,\gf \in \dom A,
	\end{align}
	and define the inclusion operator 
	\linenopax
	\begin{align}\label{eqn:Ji}
		J: \sH_A \hookrightarrow \sH, \qq\text{by}\qq J\gf = \gf, \q\text{for } \gf \in \sH_A.
	\end{align}

\begin{defn}\label{def:Friedrichs-extension}
	If $A$ is symmetric and nonnegative densely-defined operator, then the \emph{Friedrichs extension of $A$} is the operator $A_\sF$ with
	\linenopax
	\begin{align}\label{eqn:Fried}
		\dom A_\sF := \dom A^\ad \cap J(\sH_A),
		\q\text{and}\q
		A_\sF \gf = A \gf, \text{ for } \gf \in \dom A.
	\end{align}
	Here, as usual,
	\linenopax
	\begin{align}\label{eqn:dom(adj)}
		\dom A^\ad := \{\gy \in \sH \suth \exists C<\iy \text{ with } |\la \gy,A\gf \ra| \leq C \|\gf\|, \text{ for all } \gf \in \dom A\}.
	\end{align}

\end{defn}

\begin{theorem}\label{thm:Fried}
	The operator $(JJ^\ad)^{-1}$ is the Friedrichs extension of $A$.
\end{theorem}
\begin{proof}
	(1) We first show that the operator $(JJ^\ad)^{-1}$ is a self-adjoint extension of $A$.
	The inclusion operator $J: \sH_A \hookrightarrow \sH$ is contractive because the estimate \eqref{eqn:coercivity} implies 
	\linenopax
	\begin{align}\label{eqn:norm-of-J}
		\|Jf\| = \|f\| \leq \|f\|_A, \qq \text{for all } f \in \sH.
	\end{align}
	From general theory, we know that $\|J^\ad\| = \|J\|$, so both $J$ and $J^\ad$ are contractive with respect to their respective norms, and hence $JJ^\ad:\sH \to \sH$ is also contractive. We deduce that $JJ^\ad$ is a contractive self-adjoint operator in \sH.
	
	Using the self-adjointness of $JJ^\ad$ and definitions \eqref{eqn:A-norm}--\eqref{eqn:Ji}, we have that the following holds for any $\gy, \gf \in \dom A$:
	\linenopax
	\begin{align}\label{eqn:JJad-identity}
		\la \gy, JJ^\ad A\gf \ra
		= \la JJ^\ad \gy, A\gf \ra
		= \la J^\ad \gy, A\gf \ra
		= \la J^\ad \gy, \gf \ra_A
		= \la \gy, J \gf \ra
		= \la \gy, \gf \ra. 
	\end{align}
	Since \eqref{eqn:JJad-identity} holds on the dense subset $\dom A$, we have
	\linenopax
	\begin{align}\label{eqn:JJad-inv-identity}
		JJ^\ad A \gf = \gf \q\text{for any}\q \gf \in \dom A.
	\end{align}
	and it follows immediately that $JJ^\ad$ is invertible on $\ran A$.
	\emph{A fortieri}, the identity \eqref{eqn:JJad-inv-identity} shows that $(JJ^\ad)^{-1}$ is an extension of $A$.

	(2) Next, we must show that $\ran JJ^\ad = \dom A^\ad \cap J(\sH_A)$.
	Let $\gy \in \ran JJ^\ad$. Then $\gy = JJ^\ad \gf$ for some $\gf \in \sH_A$, so $y \in J(\sH_A)$ is immediate. To see that $\gy \in \dom A^\ad$, note that for any $\gf \in \dom A$, part (1) of this proof gives
	\linenopax
	\begin{align}\label{eqn:psi-in-dom(adj)}
		|\la \gy, A\gf \ra| 
		= |\la JJ^\ad \gx, A\gf \ra| 
		= |\la \gx, JJ^\ad A\gf \ra| 
		= |\la \gx, \gf \ra| 
		\leq \|\gx\| \|\gf\|,
	\end{align}
	so the bound in \eqref{eqn:dom(adj)} is satisfied with $C = \|\gx\|$. This shows $\ran JJ^\ad \ci \dom A^\ad \cap J(\sH_A)$.
	
	Now for $\gy \in \dom A^\ad \cap J(\sH_A)$, we will prove the reverse containment. Since $\gy \in \dom A$, we have $\gy = JJ^\ad A\gy$ by part (1), so $\gy \in \ran JJ^\ad$.
\end{proof}

\begin{defn}
	A symmetric operator $A$ is \emph{semibounded} iff there is some $c>-\iy$ for which 
	\linenopax
	\begin{align}\label{eqn:semibounded}
		\la \gf, A\gf \ra
		\geq c \la \gf, \gf\ra, \q\text{for all } \gf \in \dom A.
	\end{align}
\end{defn}
	
\begin{defn}\label{def:Friedrichs-extension-semibounded}
	If $A$ is semibounded, then $A+c+1$ is a symmetric and nonnegative densely-defined operator satisfying \eqref{eqn:coercivity}, and the Friedrichs extension procedure may be applied to construct $(A+c+1)_\sF$ as in Definition~\ref{def:Friedrichs-extension}. 
The \emph{Friedrichs extension of $A$} is thus defined
	\linenopax
	\begin{align}\label{eqn:AF}
		A_\sF := (A+c+1)_\sF-c-1.
	\end{align}
\end{defn}

\begin{remark}\label{rem:}
	 While there are already several constructions of Friedrichs' extension (and corresponding proofs), we feel that our Theorem~\ref{thm:Fried} has attractive features, both novelty and simplicity. For example, Kato's approach \cite[\S2.3]{Kat95} depends on first developing a rather elaborate theory of closable forms, while by contrast, our proof is simple and direct. Additionally, the tools developed here are precisely those which we need in our analysis of the network Laplacian as a semibounded Hermitian operator with dense domain in \HE, the Hilbert space of functions of finite energy on a graph. For readers interested in earlier approaches to Friedrichs' extension, we refer to, for example the books by Dunford-Schwartz \cite{DuSc88}, by Kato \cite{Kat95}, and by Reed \& Simon \cite{ReedSimonI}. The following corollary shows that part of Kato's results can be recovered from Lemma~\ref{thm:closable-pair} and Theorem~\ref{thm:Fried}.
\end{remark}

\begin{cor}
	For a given Hilbert space \sH, there is a bijective correspondence between the collection of densely-defined closed quadratic forms $q$ which satisfy $q(\gf,\gf) \geq \|\gf\|^2$, and the collection of self-adjoint operators $A$ on \sH which satisfy $A \geq 1$. More precisely:
	\begin{enumerate}
	\item Given $A$, let $\dom q := \dom A^{1/2}$ and define
		\begin{align*}
			q(\gf, \gy) := \la A^{1/2} \gf, A^{1/2} \gf \ra, \qq \forall \gf,\gy \in \dom q.
		\end{align*}
	\item Given $q$, let $J: \dom q \to \sH$ be the inclusion map $J \gf = \gf$ and define $A := (JJ^\ad)^{-1}$.
	\end{enumerate}
\end{cor}
\begin{proof}
	The proof of (1) is straightforward; the nontrivial direction of the correspondence is (2), but this follows immediately from Theorem~\ref{thm:Fried}.
\end{proof}

\section{A generalization of the Krein construction}

The following result is used to generalize some results of \cite{Krein}. It also offers a more streamlined proof; see Corollary~\ref{thm:sa-extensions} and Remark~\ref{rem:Krein}.

\begin{theorem}\label{thm:Krein-generalization}
	Suppose that $\sH_1$ and $\sH_2$ are Hilbert spaces with $\sD \ci \sH_1 \cap \sH_2$, and that $\sD$ is dense in $\sH_1$ (but not necessarily in $\sH_2$). Define $\sD^\ad \ci \sH_2$ by 
	\begin{align}\label{eqn:D*}
		\sD^\ad := \{h \in \sH_2 \suth \exists C \in(0,\iy) 
			\text{ for which } |\la \gf, h \ra_{\sH_2}|  \leq C\|\gf\|_{\sH_1}, 
			\;\forall \gf \in \sD\}.
	\end{align}
	Then $\sD^\ad$ is dense in $\sH_2$ if and only if there exists a self-adjoint operator \gL in $\sH_1$ with $\sD \ci \dom \gL$ and 
	\begin{align}\label{eqn:gL-identity}
		\la \gf, \gL \gf\ra_{\sH_1} = \|\gf\|_2^2, \qq \text{for all } \gf \in \sD.
	\end{align}
\end{theorem}
\begin{proof}
	Consider the inclusion operator $J:\sH_1 \to \sH_2$ given by
	\begin{align*}
		\dom J = \sD, \qq J \gf = \gf, \q \gf \in \sD.
	\end{align*}
	By the definition of $\dom J^\ad$, we know that $h \in \dom J^\ad$ iff there is a finite $C=C_h$ such that $|\la J\gf, h\ra_{\sH_2}| \leq C\|\gf\|_{\sH_1}$. By \eqref{eqn:D*}, this means $h \in \dom J^\ad$ iff $h \in \sD^\ad$, i.e., that $\dom J^\ad = \sD^\ad$. Consequently, the assumption that \eqref{eqn:D*} is dense in $\sH_2$ is equivalent to $J^\ad$ being densely defined, and hence also equivalent to $J$ being closable. By a theorem of von Neumann, the operator $\gL := J^\ad \cj{J}$ is self-adjoint in $\sH_1$. Now for $\gf \in \sD$, we have
	\begin{align*}
		\la \gf, \gL \gf\ra_{\sH_1} 
		&= \la \gf, J^\ad \cj{J} \gf\ra_{\sH_1} 
		= \la \cj{J} \gf, \cj{J} \gf\ra_{\sH_2} 
		= \la J\gf, J\gf \ra_{\sH_2} 
		= \la \gf, \gf \ra_{\sH_2} 
		= \|\gf\|_2^2,
	\end{align*}
	which verifies \eqref{eqn:gL-identity}.
	
	For the converse, we need to show that $\sD^\ad$ is dense in $\sH_2$. To this end, we exhibit a set $\sV \ci \sD^\ad \ci \sH_2$, with $\sV$ dense in $\sH_2$. Note that \eqref{eqn:gL-identity} implies the existence of a well-defined partial isometry $K: \sH_1 \to \sH_2$ given by $K\gL^{1/2} \gf = \gf$, $\forall \gf \in \sD$, and satisfying $\dom K = K^\ad K = \cj{\ran(\gL^{1/2})}$. We extend $K$ by defining $K=0$ on $(\dom K)^\perp$, and then defining $\sV := \{\gy \in \sH_2 \suth K^\ad \gy \in \dom (\gL^{1/2})\}$. For $\gy \in \sV$, the definition of $K$ and the Cauchy-Schwarz inequality now yield
	\begin{align*}
		\left| \la \gy, \gf \ra_2 \right|
		= \left| \la \gy, K \gL^{1/2}\gf \ra_2 \right|
		= \left| \la \gL^{1/2} K^\ad \gy,  \gf \ra_1 \right|
		\leq \left\| \gL^{1/2} K^\ad \gy\right\|_1 \left\| \gf \right\|_1,
	\end{align*}
	for every $\gf \in \sD$, whence $\sV \ci \sD^\ad$. Since $\gL^{1/2}$ is densely defined, $\sV$ is dense in $\sH_2$.
\end{proof}

Example~\ref{exm:trivial-adjoint} illustrates the relationship that can exist between the adjoint of an operator between $L^2$ spaces, and the Radon-Nikodym derivative of their respective measures, and how mutual orthogonality of these measures can cause a catastrophic failure of the adjoint. We return to this theme in the following example, which shows how our main result Theorem~\ref{thm:Krein-generalization} can be regarded as a noncommutative version of the Lebesgue-Radon-Nikodym decomposition. We pursue this line of enquiry further in the forthcoming work \cite{Charp}; see also \cite{Jor80b, Ota88, Hassi07}.

\begin{exm}\label{exm:RN}
	Let $(X,\sA)$ be a measure space on which two regular, positive, and \gs-finite measures $\gm_1$ and $\gm_2$ are defined. Let $\sH_i := L^2(X,\gm_i)$ for $i=1,2$, and let $\sD := C_c(X)$. Then the equivalent conditions in the conclusion of Theorem~\ref{thm:Krein-generalization} hold if and only if $\gm_2 \ll \gm_1$. In this case, $\gL$ corresponds to multiplication by the Radon-Nikodym derivative $f:=\frac{d\gm_2}{d\gm_1}$, and \eqref{eqn:gL-identity} can be written
	\begin{align*}
		\la \gf, \gL\gf\ra_1 
		= \int_X \cj\gf \gf f \dgm[1]
		= \int_X |\gf|^2 \frac{d\gm_2}{d\gm_1} \dgm[1]
		= \int_X |\gf|^2 \dgm[2]
		= \|\gf\|_2^2,
		\qq\forall \gf \in C(X).
	\end{align*}
\end{exm}

The connection between is made precise in general by the spectral theorem, in the following corollary of Theorem~\ref{thm:Krein-generalization}.
\begin{cor}\label{thm:RN-connection}
	Assume the hypotheses of Theorem~\ref{thm:Krein-generalization}. Then, for every $\gf \in \sD$, there is a Borel measure $\gm_\gf$ on $[0,\iy)$ such that
	\begin{align}\label{eqn:spec-meas}
		\|\gf\|_1^2 = \gm_\gf([0,\iy))
		\qq \text{and} \qq
		\|\gf\|_2^1 = \int_0^\iy \gl \dgm[\gf](\gl).
	\end{align}
\end{cor}
\begin{proof}
	Following the proof of Theorem~\ref{thm:Krein-generalization}, we take $J:\sD \to \sH_2$ by $J\gf=\gf$ and obtain the self-adjoint operator $\gL = J^\ad J$. The spectral theorem yields a spectral resolution
	\begin{align*}
		\gL = \int_0^\iy \gl E_\gL(d\gl),
	\end{align*}
	where $E_\gL$ is the associated projection-valued measure. If we define $\gm_\gf$ via
	\begin{align*}
		\dgm[\gf] := \|E_\gL(\dgl) \gf\|_1^2,
	\end{align*}
	then the conclusions in \eqref{eqn:spec-meas} follow from the spectral theorem.
\end{proof}

For an additional application of Theorem~\ref{thm:Krein-generalization}, see the example of the Laplace operator on the energy space given in Example~\ref{exm:l2-HE}.

\section{Application: Laplace operators on infinite networks}
\label{sec:Laplace-operators-on-infinite-networks}

We now proceed to introduce the key notions used throughout this paper: resistance networks, the energy form \energy, the Laplace operator \Lap, and their elementary properties. For further background, we refer to \cite{DGG, ERM, bdG, RBIN, SRAMO, Multipliers, Comparisons, UnboundedCont, OTERN, Lyons, KellerLenz09, KellerLenz10}.

\begin{defn}\label{def:network}
  A \emph{(resistance) network} $(\Graph,\cond)$ is a connected weighted undirected graph with vertex set \Graph and adjacency relation defined by a symmetric \emph{conductance function} $\cond: \Graph \times \Graph \to [0,\iy)$. 
  More precisely, there is an edge connecting $x$ and $y$ iff $c_{xy}>0$, in which case we write $x \nbr y$. The nonnegative number $c_{xy}=c_{yx}$ is the weight associated to this edge and it is interpreted as the conductance, or reciprocal resistance of the edge.
  
  We make the standing assumption that $(\Graph,\cond)$ is \emph{locally finite}. This means that every vertex has \emph{finite degree}, i.e., for any fixed $x \in \Graph$ there are only finitely many $y \in \Graph$ for which $c_{xy}>0$. We denote the net conductance at a vertex by 
  \linenopax
  \begin{align}\label{eqn:c(x)}
      \cond(x) := \sum_{y \nbr x} \cond_{xy}.     
  \end{align}
Motivated by current flow in electrical networks, we also assume $c_{xx}=0$ for every vertex $x \in G$.  

In this paper, \emph{connected} means simply that for any $x,y \in \Graph $, there is a finite sequence $\{x_i\}_{i=0}^n$ with $x=x_0$, $y=x_n$, and $\cond_{x_{i-1} x_i} > 0$, $i=1,\dots,n$.  
For any network, one can fix a reference vertex, which we shall denote by $o$ (for ``origin''). It will always be apparent that our calculations depend in no way on the choice of $o$.
\end{defn}

\begin{defn}\label{def:graph-laplacian}
  The \emph{Laplacian} on \Graph is the linear difference operator which acts on a function $u:\Graph \to \bR$ by
  \linenopax
  \begin{equation}\label{eqn:Lap}
    (\Lap u)(x) :
    = \sum_{y \nbr x} \cond_{xy}(u(x)-u(y)).
  \end{equation}
  A function $u:\Graph \to \bR$ is \emph{harmonic} iff $\Lap u(x)=0$ for each $x \in \Graph$.
  Note that the sum in \eqref{eqn:Lap} is finite by the local finiteness assumption above, and so the Laplacian is well-defined.
\end{defn}

The domain of \Lap, considered as an operator on \HE or $\ell^2(G)$, is given in Definition~\ref{def:domLapE} and Definition~\ref{def:domLap2}.

\begin{defn}\label{def:graph-energy}
  The \emph{energy form} is the (closed, bilinear) Dirichlet form
  \linenopax
  \begin{align}\label{eqn:def:energy-form}
    \energy(u,v)
    := \frac12 \sum_{x,y \in \Graph} \cond_{xy}(u(x)-u(y))(v(x)-v(y)),
  \end{align}
  which is defined whenever the functions $u$ and $v$ lie in the domain
  \linenopax
  \begin{equation}\label{eqn:def:energy-domain}
    \dom \energy = \{u:\Graph \to \bR \suth \energy(u,u)<\iy\}.
  \end{equation}
  Hereafter, we write the energy of $u$ as $\energy(u) := \energy(u,u)$. Note that $\energy(u)$ is a sum of nonnegative terms and hence converges iff it converges absolutely. 
\end{defn}

The energy form \energy is sesquilinear and conjugate symmetric on $\dom \energy$ and would be an inner product if it were positive definite. Let \one denote the constant function with value 1 and observe that $\ker \energy = \bR \one$. One can show that $\dom \energy / \bR \one$ is complete and that \energy is closed;  see \cite{DGG,OTERN}, \cite{Kat95}, or \cite{FOT94}.

\begin{defn}\label{def:H_energy}\label{def:The-energy-Hilbert-space}
  The \emph{energy (Hilbert) space} is $\HE := \dom \energy / \bR \one$. The inner product and corresponding norm are denoted by
  \linenopax
  \begin{equation}\label{eqn:energy-inner-product}
    \la u, v \ra_\energy := \energy(u,v)
    \q\text{and}\q
    \|u\|_\energy := \energy(u,u)^{1/2}.
  \end{equation}
\end{defn}

It is shown in \cite[Lem.~2.5]{DGG} that the evaluation functionals $L_x u = u(x) - u(o)$ are continuous, and hence correspond to elements of \HE by Riesz duality (see also \cite[Cor.~2.6]{DGG}). When considering \bC-valued functions, \eqref{eqn:energy-inner-product} is modified as follows: $\la u, v \ra_\energy := \energy(\cj{u},v)$.

\begin{defn}\label{def:vx}\label{def:energy-kernel}
  Let $v_x$ be defined to be the unique element of \HE for which
  \linenopax
  \begin{equation}\label{eqn:v_x}
    \la v_x, u\ra_\energy = u(x)-u(o),
    \qq \text{for every } u \in \HE.
  \end{equation}
  Note that $v_o$ corresponds to a constant function, since $\la v_o, u\ra_\energy = 0$ for every $u \in \HE$. Therefore, $v_o$ may be safely omitted in some calculations. 
\end{defn}

  As \eqref{eqn:v_x} means that the collection $\{v_x\}_{x \in \Graph}$ forms a reproducing kernel for \HE, we call $\{v_x\}_{x \in \Graph}$ the \emph{energy kernel}. It follows that the energy kernel has dense span in \HE; cf. \cite{Aronszajn50}.\footnote{To see this, note that a RKHS is a Hilbert space $H$ of functions on some set $X$, such that point evaluation by points in $X$ is continuous in the norm of $H$. Consequently, every $x \in X$ defines a vector $k_x \in H$ by Riesz's Theorem, and it is immediate from this that $\spn\{k_x\}_{x \in X}$ is dense in $H$.}

\begin{remark}[Differences and representatives]\label{rem:differences}
  Equation \eqref{eqn:v_x} is independent of the choice of representative of $u$ because the right-hand side is a difference: if $u$ and $u'$ are both representatives of the same element of \HE, then $u' = u+k$ for some $k \in \bR$ and
  $u'(x) - u'(o) = (u(x)+k)-(u(o)+k) = u(x)-u(o).$
  By the same token, the formula for \Lap given in \eqref{eqn:Lap} describes unambiguously the action of \Lap on equivalence classes $u \in \HE$. Indeed, formula \eqref{eqn:Lap} defines a function $\Lap u:\Graph \to \bR$ but we may also interpret $\Lap u$ as the class containing this representative.
\end{remark}

\begin{defn}\label{def:d_x}
  Let $\gd_x \in \ell^2(G)$ denote the Dirac mass at $x$, i.e., the characteristic function of the singleton $\{x\}$ and let $\gd_x \in \HE$ denote the element of \HE which has $\gd_x \in \ell^2(G)$ as a representative. The context will make it clear which meaning is intended. 
\end{defn}

\begin{remark}\label{rem:Diracs-in-HE}
	 Observe that $\energy(\gd_x) = \cond(x) < \iy$ is immediate from \eqref{eqn:def:energy-form}, and hence one always has $\gd_x \in \HE$ (recall that $c(x)$ is the total conductance at $x$; see \eqref{eqn:c(x)}).
\end{remark}

\begin{defn}\label{def:Fin}
  For $v \in \HE$, one says that $v$ has \emph{finite support} iff there is a finite set $F \ci G$ such that $v(x) = k \in \bC$ for all $x \notin F$. Equivalently, the set of functions of finite support is 
  \linenopax
  \begin{equation}\label{eqn:span(dx)}
    \spn\{\gd_x\} = \{u \in \dom \energy \suth u \text{ is constant outside some finite set}\}.
  \end{equation}
  Define \Fin to be the \energy-closure of $\spn\{\gd_x\}$. 
\end{defn}

\begin{defn}\label{def:Harm}
  The set of harmonic functions of finite energy is denoted
  \linenopax
  \begin{equation}\label{eqn:Harm}
    \Harm := \{v \in \HE \suth \Lap v(x) = 0, \text{ for all } x \in G\}.
  \end{equation}
\end{defn}

The following result is well known; see \cite[\S{VI}]{Soardi94}, \cite[\S9.3]{Lyons}, \cite[Thm.~2.15]{DGG}, or the original \cite[Thm.~4.1]{Yamasaki79}.

\begin{theorem}[Royden Decomposition]\label{thm:HE=Fin+Harm}
  $\HE = \Fin \oplus \Harm$.
\end{theorem}
 

\begin{defn}\label{def:dipole}
  A \emph{monopole} is any $w \in \HE$ satisfying the pointwise identity $\Lap w = \gd_x$ (in either sense of Remark~\ref{rem:differences}) for some vertex $x \in \Graph$. 
  A \emph{dipole} is any $v \in \HE$ satisfying the pointwise identity $\Lap v = \gd_x - \gd_y$ for some $x,y \in \Graph$.%
\end{defn}

\begin{remark}\label{rem:monotransience}
  It is easy to see from the definitions (or \cite[Lemma~2.13]{DGG}) that energy kernel elements are dipoles, i.e., that $\Lap v_x = \gd_x - \gd_o$, and that one can therefore always find a dipole for any given pair of vertices $x,y \in G$, namely, $v_x-v_y$. On the other hand, monopoles exist if and only if the network is transient (see \cite[Thm.~2.12]{Woess00} or \cite[Rem.~3.5]{DGG}). 
\end{remark}  

\begin{remark}\label{rem:normalization}
  Denote the unique energy-minimizing monopole at $o$ by $w_o$; the existence of such an object is explained in \cite[\S3.1]{DGG}.
  We are interested in the family of monopoles defined by
  \linenopax
  \begin{align}\label{eqn:w_x}
    \monov := w_o + v_x, \qq x \neq o.
  \end{align}
  We use the representatives specified by
  \linenopax
  \begin{align}\label{eqn:w_x(o)}
    \monov(y) = \la \monov,\monoy\ra_\energy = \monoy(x), 
    \qq\text{and}\qq
    v_x(o)=0.
  \end{align}
  When $\Harm=0$, $\energy(\monov) = \la \monov,\monov\ra_\energy = \monov(x)$ is the \emph{capacity} of $x$; see, e.g., \cite[\S4.D]{Woess09}.
\end{remark}

\begin{lemma}[{\cite[Lem.~2.11]{DGG}}]
  \label{thm:<delta_x,v>=Lapv(x)}
  For $x \in \Graph$ and $u \in \HE$,  $\la \gd_x, u \ra_\energy = \Lap u(x)$.
  \begin{proof}
    Compute $\la \gd_x, u \ra_\energy = \energy(\gd_x, u)$ directly from formula \eqref{eqn:def:energy-form}.
  \end{proof}
\end{lemma}

\begin{lemma}\label{thm:Lap-mono-Kron}
  For any $x,y \in G$, 
  \linenopax
  \begin{align}\label{eqn:Lap-mono-Kron}
    \Lap \monov(y) = \Lap \monoy(x) 
    = \la \monov,\Lap \monoy\ra_\energy 
    = \la \Lap \monov, \monoy\ra_\energy = \gd_{xy},  
  \end{align}
  where $\gd_{xy}$ is the Kronecker delta.
  \begin{proof}
    First, note that $\Lap \monov(y) = \gd_{xy} = \Lap \monov(y)$ as functions, immediately from the definition of monopole. Then the substitution $\Lap \monoy = \gd_y$ gives 
    \linenopax
  	\begin{align}\label{eqn:<w,d>=d}
      \la \monov,\Lap \monoy\ra_\energy
      = \la \monov,\gd_y\ra_\energy
      = \Lap \monov(y)
    \end{align}
    by Lemma~\ref{thm:<delta_x,v>=Lapv(x)}, and similarly for the other identity. 
  \end{proof}
\end{lemma}

\subsection{\Lap as an unbounded operator} 
\label{sec:ell2(G)}
In this subsection, we consider closability and apply the results of the earlier sections. 
As there are many uses of the notation $\ell^2(G)$, we provide the following elementary definitions to clarify our conventions.

\begin{defn}\label{def:ell2}
For functions $u,v:G \to \bR$, define the inner product
  \linenopax
  \begin{align}\label{eqn:ell2-inner-product}
    \la u, v\ra_2 := \sum_{x \in G} u(x) v(x).
  \end{align}
\end{defn}
  
\begin{defn}\label{def:domLap2}
  The closed operator $\Lap_2$ on $\ell^2(G)$ is obtained by taking the graph closure (see Remark~\ref{rem:domLap2}) of the operator \Lap which is defined pointwise by \eqref{eqn:Lap} on $\spn\{\gd_x\}_{x \in G}$, the subspace of (finite) linear combinations of point masses. 
\end{defn}

\begin{defn}\label{def:domLapE}
  The closed operator \LapE on \HE is obtained by taking the graph closure of the operator \Lap defined on $\spn\{\monov\}_{x \in G}$ pointwise by \eqref{eqn:Lap}. 
  %
\end{defn}

\begin{remark}\label{rem:domLap2}
  It is shown in \cite[Lem.~2.7 and Thm.~2.8]{SRAMO} states that $\Lap$ is semibounded and essentially self-adjoint as an operator on $\spn\{\gd_x\}_{x \in G}$. It follows that $\Lap$ is closable by the same arguments as in the end of the proof of Lemma~\ref{thm:semibounded}, whence $\Lap_2$ is closed, self-adjoint, and in particular, well-defined. However, closability will also follow in this context from the properties of symmetric pairs shown in Lemma~\ref{thm:closable-pair}.
  Note that in sharp contrast, the analogous operator \LapE is not automatically self-adjoint (see \cite{SRAMO}) and hence some care is needed (for example, in the proof of Lemma~\ref{thm:semibounded}).
  See also \cite{Woj07, KellerLenz09, KellerLenz10}.
\end{remark}

The following lemma shows that Definition~\ref{def:domLapE} makes sense.  

 
\begin{lemma}\label{thm:semibounded}
   \LapE is a well-defined, non-negative, closed and Hermitian operator on \HE. 
  \begin{proof}
    Let $\gx = \sum_{x \in F} \gx_x \monov$, for some finite set $F \ci G$. By \eqref{eqn:Lap-mono-Kron}, 
    \linenopax
    \begin{align}\label{eqn:semibounded}
      \la u, \Lap u \ra_\energy
      &= \sum_{x,y \in F} {\gx_x} \gx_y \la \monov,\Lap \monoy\ra_\energy 
       = \sum_{x,y \in F} {\gx_x} \gx_y \gd_{xy}
       = \sum_{x \in F} |\gx_x|^2 \geq 0.  
    \end{align} 
    Since the conductance function $c$ is \bR-valued, the Laplacian commutes with conjugation and therefore is also symmetric as an operator in the corresponding \bC-valued Hilbert space. This implies \Lap is Hermitian and hence contained in its adjoint. Since every adjoint operator is closed, \Lap is closable. Furthermore, the closure of any semibounded operator is semibounded.    
    To see that the image of \Lap lies in \HE, note from Lemma~\ref{thm:Lap-mono-Kron} that $\Lap\monov = \gd_x \in \HE$ by Remark~\ref{rem:Diracs-in-HE}.
  \end{proof}
\end{lemma}


In the following theorem, we apply Lemma~\ref{thm:closable-pair} to the construction laid out in \cite{Krein}. This shows how one can recover the closability results described in Remark~\ref{rem:domLap2} and Lemma~\ref{thm:semibounded} in a manner which is both quicker and more elegant.

\begin{theorem}\label{thm:KL-closable}
	Define $K: \spn\{\gd_x\}_{x \in G} \to \HE$ by $K\gd_x = \gd_x$ and define $L: \spn\{v_x\}_{x \in G} \to \ell^2(G)$ by $L(v_x) = \gd_x - \gd_o$. Then
	\begin{align}\label{eqn:KL-closable}
		\la K \gf, \gy \ra_\energy = \la \gf, L\gy\ra_2,
		\qq \text{for all } \gf \in \spn\{\gd_x\}_{x \in G}, \gy \in \spn\{v_x\}_{x \in G}.
	\end{align}
	It therefore follows immediately from Lemma~\ref{thm:closable-pair} that both operators are closable.
\end{theorem}
\begin{proof}
	It suffices to establish \eqref{eqn:KL-closable} for every $\gf=\gd_x$ and $\gy = v_y$, and
	\begin{align*}
		\la K \gd_x, v_y \ra_\energy 
		= \la \gd_x, v_y \ra_\energy 
		= \gd_x(y) - \gd_x(o)
		= \la \gd_x, \gd_y - \gd_o\ra_2
		= \la \gd_x, Lv_y\ra_2
	\end{align*}
	follows from the reproducing property \eqref{eqn:v_x}.
\end{proof}

This provides a more effective way of proving a key result of \cite{Krein}, in the following corollary. \begin{cor}\label{thm:sa-extensions}
	In the notation of Theorem~\ref{thm:KL-closable}, $K^\ad\cj{K}$ is a self-adjoint extension of $\Lap_2$ and $L^\ad\cj{L}$ is a self-adjoint extension of $\LapE$. 
\end{cor}

\begin{remark}\label{rem:Krein}
	In \cite{SRAMO}, it is shown that $\Lap_2$ is essentially self-adjoint, from which it follows that $K^\ad\cj{K}$ is the \emph{unique} self-adjoint extension of $\Lap_2$. It is shown in \cite{Krein} that $L^\ad \cj{L}$ is the \emph{Krein extension} of $\LapE$.
\end{remark}

\begin{proof}[Proof of Corollary~\ref{thm:sa-extensions}]
	Self-adjointness of these operators follows by a celebrated theorem of von Neumann once closability is established (which is given by Theorem~\ref{thm:KL-closable}). To establish that $\LapE \ci L^\ad\cj{L}$, note that the definitions give
	\begin{align*}
		\la v_y, L^\ad \cj{L} v_x\ra_\energy
		= \la Lv_y, L v_x\ra_2
		&= \la \gd_y-\gd_o, \gd_x-\gd_o\ra_2 \\
		&= (\gd_x-\gd_o)(y) - (\gd_x-\gd_o)(o) 
		= \la v_y, \gd_x-\gd_o\ra_\energy,
	\end{align*}
	for any $v_x$. This shows that the action of $L^\ad \cj{L}$ agrees with $\LapE$ on $\dom \LapE$.
\end{proof}

\begin{exm}\label{exm:l2-HE}
	If we take $\sH_1 = \ell^2(G)$, $\sH_2 = \HE$, and $\sD = \spn\{\gd\}_{x \in G}$, then the hypotheses of Theorem~\ref{thm:Krein-generalization} are satisfied. The only detail requiring effort to check is that $\spn\{v_x\}_{x \in G} \ci \sD^\ad$ (whence $\sD^\ad$ is dense in $\sH_2$). To see this, note that the reproducing property of $v_x$ gives
	\begin{align*}
		|\la \gf, v_x \ra_{\HE}| 
		&= |\gf(x) - \gf(o)|
		= |\la \gf, \gd_x - \gd_o\ra_{\ell^2}|
		\leq \|\gf\|_{\ell^2} \|\gd_x - \gd_o\|_{\ell^2} 
		= \sqrt2 \|\gf\|_{\ell^2},
	\end{align*}
	so one can take $C = 2$ in \eqref{eqn:D*}. In this case, the operator \gL is $\LapK$, the Krein extension of the energy Laplacian; see \cite{Krein, KellerLenz10}.
\end{exm}

{\small
\bibliographystyle{math}
\bibliography{networks}
}

\end{document}